\documentclass[12pt]{amsart}%
\usepackage{amsfonts}
\usepackage{mathtools}
\usepackage{amsmath}
\usepackage{amssymb}
\usepackage{amsthm}
\usepackage{newlfont}
\usepackage{graphicx}
\usepackage{amssymb,amsmath,xcolor,graphicx,xspace,colortbl, rotating}
\usepackage{textcomp}%
\providecommand{\U}[1]{\protect\rule{.1in}{.1in}}
\graphicspath{{proof_latest teX_graphics/}{proof_latest teX_tcache/}{proof_latest teX_gcache/}}
\DeclareGraphicsExtensions{.pdf,.eps,.ps,.png,.jpg,.jpeg}
\newtheorem {theorem}{Theorem}[section]

\newtheorem {corollary}[theorem]{Corollary}

\newtheorem {definition}[theorem]{Definition}

\newtheorem {lemma}[theorem]{Lemma}

\newtheorem {proposition}[theorem]{Proposition}

\numberwithin{equation}{section}
\begin{document}
\title[Regularity Bootstrapping]{Regularity bootstrapping for fourth order nonlinear elliptic equations}
\author{Arunima Bhattacharya AND Micah Warren }

\begin{abstract}
We consider nonlinear fourth order elliptic equations of double divergence
type. We show that for a certain class of equations where the nonlinearity is
in the Hessian, solutions that are $C^{2,\alpha}$ enjoy interior estimates on
all derivatives.

\end{abstract}
\maketitle

\section{ Introduction}

In this paper, we develop Schauder and bootstrapping theory for solutions to
fourth order non linear elliptic equations of the following double divergence
form
\begin{equation}
\int_{\Omega}a^{ij,kl}(D^{2}u)u_{ij}\eta_{kl}dx=0,\text{ }\forall\eta\in
C_{0}^{\infty}(\Omega) \label{eq1}%
\end{equation}
in $B_{1}=B_{1}(0).$ For the Schauder theory, we require the standard
Legendre-Hadamard ellipticity condition,%
\begin{equation}
a^{ij,kl}(D^{2}u(x))\xi_{ij}\xi_{kl}\geq\Lambda|\xi_{rs}|^{2} \label{LH}%
\end{equation}
while in order to bootstrap, we will require the following condition:
\begin{equation}
b^{ij,kl}(D^{2}u(x))=a^{ij,kl}(D^{2}u(x))+\frac{\partial a^{pq,kl}}{\partial
u_{ij}}(D^{2}u(x))u_{pq}(x) \label{Bdef}%
\end{equation}
satisfies%
\begin{equation}
b^{ij,kl}(D^{2}u(x))\xi_{ij}\xi_{kl}\geq\Lambda_{1}\left\Vert \xi\right\Vert
^{2}. \label{Bcondition}%
\end{equation}
Our main result is the following: Suppose that conditions (\ref{eq1}) and
(\ref{Bcondition}) are met on some open set~$U\subseteq S^{n\times n}$ (space
of symmetric matrices). ~If $u$ is a $C^{2,\alpha}$~solution with
$D^{2}u(B_{1})\subset U$, then $u$ is smooth on the interior of the domain
$B_{1}.$ \ \ \ \ 

One example of such an equation is the Hamiltonian Stationary Lagrangian
equation, which governs Lagrangian surfaces that minimize the area functional
\begin{equation}
\int_{\Omega}\sqrt{\det(I+\left(  D^{2}u\right)  ^{T}D^{2}u)}dx \label{HS}%
\end{equation}
among potential functions $u.$ (cf. \cite{MR1202805}, \cite[Proposition
2.2]{SW03}). The minimizer satisfies a fourth order equation, that, when
smooth, can be factored into a a Laplace type operator on a nonlinear
quantity. Recently in \cite{CW}, it is shown that a $C^{2}$ solution is
smooth. The results in \cite{CW} are the combination of an initial
regularity boost, followed by applications of the second order Schauder theory
as in \cite{CC}.

More generally, for a functional $F$ on the space of matrices, one may
consider a functional of the form
\[
\int_{M}F(D^{2}u)dx.
\]
The Euler-Lagrange equation will generically be of the following
double-divergence type:
\begin{equation}
\frac{\partial^{2}}{\partial x_{i}\partial x_{j}}(\frac{\partial F}{\partial
u_{ij}}(D^{2}u))=0. \label{generic}%
\end{equation}
Equation (\ref{generic}) need not factor into second order operators, so
it may be genuinely a fourth order double-divergence elliptic type equation.
It should be noted that in general, (\ref{generic}) need not take the form
of (\ref{eq1}). It does when $F(D^{2}u)$ can be written as a function of $D^{2}u^{T}D^{2}u$ (as for example (\ref{HS})). Our results in this paper apply to a
class of Euler-Lagrange equations arising from such functionals. In
particular, we will show that if $F$ is a convex function of $D^{2}u$ and a
function of $D^{2}u^{T}D^{2}u$ (such as \ref{HS} when $\left\vert
D^{2}u\right\vert \leq1$) then $C^{2,\alpha}$ solutions will be smooth. \ 

The Schauder Theory for second order divergence and non-divergence type
elliptic equations is by now well-developed, see \cite{HL} , \cite{GT}\ and
\cite{CC}. For higher order non-divergence equations, Schauder theory is
available, see \cite{Simon}. However, for higher
order equations in divergence form, much less is known. One expects the
results to be different: For second order equations, solutions to divergence
type equations with $C^{\alpha}$ coefficients are known to be $C^{1,\alpha}$,
\cite[Theorem 3.13]{HL}, whereas for non-divergence equations, solutions will
be $C^{2,\alpha}$ \cite[Chapter 6]{GT}. Recently, Dong and Zhang \cite{DZ}
have obtained general Schauder theory results for parabolic equations (of
order $2m$) in divergence form, where the time coefficients are allowed to be
merely measurable. Their proof (like ours) is in the spirit of Campanato
techniques, but requires smooth initial conditions. Our result is aimed at
showing that weak solutions are in fact smooth. Classical Schauder theory
for general systems has been developed, \cite[Chapter 5,6 ]{Morrey66}.
However, it is non-trivial to apply the general classical results to obtain
the result we are after. Even so, it is useful\ to focus on a specific class
of fourth order double-divergence operators, and offer random access to the
non-linear Schauder theory for these cases. Regularity for fourth order
equations remains an important developing area of geometric analysis. \ 

Our proof goes as follows: We start with a~$C^{2,\alpha}$ solution of
(\ref{eq1}) whose coefficient matrix is a smooth function of the Hessian of
$u.$ We first prove that~$u\in W^{3,2}$ by taking a difference quotient of
(\ref{eq1}) and give a $W^{3,2}$ estimate of $u$ in terms of its~$C^{2,\alpha
}$ norm. Again by taking difference a quotient and using the fact that
now~$u\in W^{3,2},$ we prove that~$u\in C^{3,\alpha}$.

Next, we make a more general proposition where we prove a $W^{3,2}$ estimate
for $u\in W^{2,2}$~ satisfying a uniformly elliptic~equation of the form
\[
\int(c^{ij,kl}u_{ik}+h^{jl})\eta_{jl}dx=0
\]
in $B_{1}(0),$ where $c^{ij,kl},h^{kl}\in W^{1,2}(B_{1})$ and~$\eta$ is a test
function in $B_{1}$. Using the fact that $u\in W^{3,2},$ we prove that $u\in
C^{3,\alpha}$ and also derive a $C^{3,\alpha}$ estimate of $u$ in terms of its
$W^{3,2}$ norm. \ $\ \ $Finally, using difference quotients and dominated
convergence, we achieve all higher orders of regularity. \ 

\begin{definition}
We say an equation of the form (\ref{eq1}) is \textbf{regular on }$U\subseteq
S^{n\times n}$ $\ $when the coefficients of the equation satisfy the~
following conditions on $U$:

1. The coefficients $a^{ij,kl}$ depend smoothly on $D^{2}u$.

2.The coefficients $a^{ij,kl}$ satisfy (\ref{LH}).

3.Either $b^{ijkl}$ or $-b^{ijkl}$ (given by (\ref{Bdef})) satisfy
(\ref{Bcondition}).
\end{definition}

The following is our main result.

\begin{theorem}
\textbf{~~~~~}Suppose that $u\in C^{2,\alpha}(B_{1})$ satisfies the following
fourth order equation%
\begin{align*}
\int_{B_{1}(0)}a^{ij,kl}(D^{2}u(x))u_{ij}(x)\eta_{kl}(x)dx  &  =0\\
\forall\eta &  \in C_{0}^{\infty}(B_{1}(0))
\end{align*}
\textit{If }$a^{ij,kl}$ is regular on an open set containing $D^{2}%
u(B_{1}(0)),$ then $u$ is smooth on $B_{r}(0)$ for $r<1$\textit{. }
\end{theorem}

To prove this, we will need the following two Schauder type estimates.

\begin{proposition}
\label{prop3}~Suppose $u\in W^{2,\infty}(B_{1})$ satisfies the following%
\begin{align}
\int_{B_{1}(0)}\left[  c^{ij,kl}(x)u_{ij}(x)+f^{kl}(x)\right]  \eta_{kl}(x)dx
&  =0\label{Cequation}\\
\forall\eta &  \in C_{0}^{\infty}(B_{1}(0))\nonumber
\end{align}
where $c^{ij,kl},f^{kl}\in W^{1,2}(B_{1}),$ and~$c^{ij,kl}$ satisfies
(\ref{LH}). Then $u\in W^{3,2}(B_{1/2})$~ and%
\[
\left\Vert D^{3}u\right\Vert _{L^{2}(B_{1/2})}\leq C(||u||_{W^{2,\infty}%
(B_{1})},\left\Vert f^{kl}\right\Vert _{W^{1,2}(B_{1})},\left\Vert
c^{ij,kl}\right\Vert _{W^{1,2}},\Lambda_{1}).
\]

\end{proposition}

\begin{proposition}
\label{prop4}Suppose $u\in C^{2,\alpha}(B_{1})$ satisfies (\ref{Cequation}) in
$B_{1}$ where $c^{ij,kl},f^{kl}\in C^{1,\alpha}(B_{1})$ and~$c^{ij,kl}$
satisfies (\ref{LH}).Then we have $u\in C^{3,\alpha}(B_{1/2})$ with%
\[
||D^{3}u||_{C^{0,\alpha}(B_{1/4})}\leq C(1+||D^{3}u||_{L^{2}(B_{3/4})})
\]
and~$C=C(|c^{ij,kl}|_{C^{\alpha}(B_{1})},\Lambda_{1},\alpha)$ is a positive constant.
\end{proposition}

We note that the above estimates are appropriately scaling invariant: \ Thus
we can use these to obtain interior estimates for a solution in the interior
of any sized domain.~~

\section{Preliminaries}

We begin by considering a constant coefficient double divergence equation. \ 

\begin{theorem}
\label{five} Suppose $w\in H^{2}(B_{r})$ satisfies the constant coefficient
equation%
\begin{align}
\int c_{0}^{ik,jl}w_{ik}\eta_{jl}dx  &  =0\label{ccoef}\\
\forall\eta &  \in C_{0}^{\infty}(B_{r}(0)).\nonumber
\end{align}
\textit{ Then for any~}$0<\rho\leq r$\textit{ there holds}%

\begin{align*}
\int_{B_{\rho}}|D^{2}w|^{2}  &  \leq C_{1}(\rho/r)^{n}||D^{2}w||_{L^{2}%
(B_{r})}^{2}\\
\int_{B_{\rho}}|D^{2}w-(D^{2}w)_{\rho}|^{2}  &  \leq C_{2}(\rho/r)^{n+2}%
\int_{B_{r}}|D^{2}w-(D^{2}w)_{r}|^{2}.
\end{align*}
Here $(D^{2}w)_{\rho}$ is the average value\ of $D^{2}w$ on a ball of radius
$\rho$.
\end{theorem}

\begin{proof}
By dilation we may consider $r=1$. We restrict our consideration to the range
$\rho\in(0,a]$ noting that the statement is trivial for $\rho\in\lbrack a,1]$
where $a$ is some constant in $(0,1/2).$

First, we note that $w$ is smooth \cite[Theorem 33.10]{Driver03}. Recall
\cite[Lemma 2, Section 4, applied to elliptic case]{DongKimARMA} : For an
elliptic $4$th order $L_{0}$
\begin{align*}
L_{0}u  &  =0\text{ on }B_{R}\\
&  \implies\left\Vert Du\right\Vert _{L^{\infty}(B_{R/4})}\leq C_{3}%
(\Lambda,n)\left\Vert u\right\Vert _{L^{2}(B_{R})}.
\end{align*}
We may apply this to the second derivatives of $w$ to conclude that
\begin{equation}
\left\Vert D^{3}w\right\Vert _{L^{\infty}(B_{a})}^{2}\leq C_{4}(\Lambda
,n)\int_{B_{1}}\left\Vert D^{2}w\right\Vert ^{2}. \label{dong}%
\end{equation}
For small enough $a<1.$~ ~~ ~~~ Now
\begin{align*}
\int_{B_{\rho}}\left\vert D^{2}w\right\vert ^{2}  &  \leq C_{5}(n)\rho
^{n}\left\Vert D^{2}w\right\Vert _{L^{\infty}(B_{a})}^{2}\\
&  =C_{5}\rho^{n}\inf_{x\in B_{a}}\sup_{y\in B_{a}}\left\vert D^{2}%
w(x)+D^{2}w(y)-D^{2}w(x)\right\vert ^{2}\\
&  \leq C_{5}\rho^{n}\inf_{x\in B_{a}}\left[  D^{2}w(x)+2a\left\Vert
D^{3}w\right\Vert _{L^{\infty}(B_{a})}\right]  ^{2}\\
&  \leq2C_{5}\rho^{n}\left[  \inf_{x\in B_{a}}\left\Vert D^{2}w(x)\right\Vert
^{2}+4a^{2}\left\Vert D^{3}w\right\Vert _{L^{\infty}(B_{a})}\right] \\
&  \leq2C_{5}\rho^{n}\left[  \frac{1}{|B_{a}|}||D^{2}w||_{L^{2}(B_{a})}%
^{2}+4a^{2}C_{4}||D^{2}w||_{L^{2}(B_{a})}^{2}\right] \\
&  \leq C_{6}(a,n)\rho^{n}||D^{2}w||_{L^{2}(B_{1})}^{2}.
\end{align*}

Similarly%
\begin{align}
\int_{B_{\rho}}\left\vert D^{2}w-(D^{2}w)_{\rho}\right\vert ^{2}  &  \leq
\int_{B_{\rho}}\left\vert D^{2}w-D^{2}w(x_{0})\right\vert ^{2}\nonumber\\
&  \leq\int_{S^{n-1}}\int_{0}^{\rho}r^{2}\left\vert D^{3}w\right\vert
^{2}r^{n-1}drd\phi\nonumber\\
&  =C_{7}\rho^{n+2}\left\vert D^{3}w\right\vert _{L^{\infty}(B_{a})}^{2}.
\label{fred}%
\end{align}
Next, observe that (\ref{ccoef}) is purely fourth order, so the equation still
holds when a second order polynomial is added to the solution. \ In
particular, we may choose
\[
D^{2}\bar{w}=D^{2}w-\left(  D^{2}w\right)  _{1}%
\]
for $\bar{w}$ also satisfying the equation. \ Then
\[
D^{3}\bar{w}=D^{3}w
\]
so
\begin{align}
\left\Vert D^{3}w\right\Vert _{L^{\infty}(B_{a})}^{2}  &  =\left\Vert
D^{3}\bar{w}\right\Vert _{L^{\infty}(B_{a})}^{2}\label{ed}\\
&  \leq C_{4}\int_{B_{1}}\left\Vert D^{2}\bar{w}\right\Vert ^{2}=C_{4}%
\int_{B_{1}}\left\Vert D^{2}w-\left(  D^{2}w\right)  _{1}\right\Vert
^{2}.\nonumber
\end{align}
We conclude from (\ref{ed}) and (\ref{fred})
\[
\int_{B_{\rho}}\left\vert D^{2}w-(D^{2}w)_{\rho}\right\vert ^{2}\leq C_{7}%
\rho^{n+2}C_{4}\int_{B_{1}}\left\Vert D^{2}w-\left(  D^{2}w\right)
_{1}\right\Vert ^{2}.
\]

\end{proof}

Next, we have a corollary to the above theorem.

\begin{corollary}
\label{Cor2} \textit{Suppose }$w$\textit{ is as in the Theorem \ref{five}.
Then for any~}$u\in H^{2}(B_{r}),$ and\textit{ for any~} $0<\rho\leq r,$
there holds%
\begin{equation}
\int_{B_{\rho}}\left\vert D^{2}u\right\vert ^{2}\leq4C_{1}(\rho/r)^{n}%
\left\Vert D^{2}u\right\Vert _{L^{2}(B_{r})}^{2}+\left(  2+8C_{1}\right)
\left\Vert D^{2}v\right\Vert _{L^{2}(B_{r})}^{2}. \label{twothree}%
\end{equation}
\textit{and}%
\begin{align}
\int_{B_{\rho}}\left\vert D^{2}u-(D^{2}u)_{\rho}\right\vert ^{2}  &
\leq4C_{2}(\rho/r)^{n+2}\int_{B_{r}}\left\vert D^{2}u-(D^{2}u)_{r}\right\vert
^{2}\label{twofive}\\
&  +\left(  8+16C_{2}\right)  \int_{B_{r}}\left\vert D^{2}v\right\vert
^{2}\nonumber
\end{align}

\end{corollary}

\begin{proof}
Let $v=u-w.$ Then (\ref{twothree}) follows from direct computation:
\begin{align*}
\int_{B\rho}|D^{2}u|^{2}  &  \leq2\int_{B_{\rho}}|D^{2}w|^{2}+2\int_{B_{\rho}%
}|D^{2}v|^{2}.\\
&  \leq2C_{1}(\rho/r)^{n}||D^{2}w||_{L^{2}(B_{r})}^{2}+2\int_{B_{r}}%
|D^{2}v|^{2}\\
&  \leq4C_{1}(\rho/r)^{n}\left[  ||D^{2}v||_{L^{2}(B_{r})}^{2}+||D^{2}%
u||_{L^{2}(B_{r})}^{2}\right]  +2\int_{B_{r}}|D^{2}v|^{2}\\
&  =4C_{1}(\rho/r)^{n}\left\Vert D^{2}u\right\Vert _{L^{2}(B_{r})}%
^{2}+2[1+4C_{1}(\rho/r)^{n}]\left\Vert D^{2}v\right\Vert _{L^{2}(B_{r})}^{2}.
\end{align*}
~

Similarly%
\begin{align*}
\int_{B\rho}\left\vert D^{2}u-(D^{2}u)_{\rho}\right\vert ^{2}  &  \leq
2\int_{B_{\rho}}\left\vert D^{2}w-(D^{2}w)_{\rho}\right\vert ^{2}%
+2\int_{B_{\rho}}\left\vert D^{2}v-(D^{2}v)_{\rho}\right\vert ^{2}\\
&  \leq2\int_{B_{\rho}}\left\vert D^{2}w-(D^{2}w)_{\rho}\right\vert ^{2}%
+8\int_{B_{\rho}}\left\vert D^{2}v\right\vert ^{2}\\
&  \leq2C_{2}(\rho/r)^{n+2}\int_{B_{r}}|D^{2}w-(D^{2}w)_{r}|^{2}%
+8\int_{B_{\rho}}\left\vert D^{2}v\right\vert ^{2}\\
&  \leq2C_{2}(\rho/r)^{n+2}\left\{
\begin{array}
[c]{c}%
2\int_{B_{r}}\left\vert D^{2}u-(D^{2}u)_{r}\right\vert ^{2}\\
+2\int_{B_{r}}\left\vert D^{2}v-(D^{2}v)_{r}\right\vert ^{2}%
\end{array}
\right\}  +8\int_{B_{r}}\left\vert D^{2}v\right\vert ^{2}\\
&  \leq4C_{2}(\rho/r)^{n+2}\int_{B_{r}}\left\vert D^{2}u-(D^{2}u)_{r}%
\right\vert ^{2}\\
&  +\left(  8+16C_{2}(\rho/r)^{n+2}\right)  \int_{B_{r}}\left\vert
D^{2}v\right\vert ^{2}.
\end{align*}
The statement follows, noting that $\rho/r\leq1.$
\end{proof}

We will be using the following Lemma frequently, so we state it here for the
reader's convenience.

\begin{lemma}
\cite[Lemma 3.4]{HL}. \ \ Let $\phi$ be a nonnegative and nondecreasing
function on $[0,R].$ \ Suppose that
\[
\phi(\rho)\leq A\left[  \left(  \frac{\rho}{r}\right)  ^{\alpha}%
+\varepsilon\right]  \phi(r)+Br^{\beta}%
\]
for any $0<\rho\leq r\leq R,$ with $A,B,\alpha,\beta$ nonnegative constants
and $\beta<\alpha.$ \ Then for any $\gamma\in(\beta,\alpha),$ there exists a
constant $\varepsilon_{0}=\varepsilon_{0}(A,\alpha,\beta,\gamma)$ such that if
$\varepsilon<\varepsilon_{0}$ we have for all $0<\rho\leq r\leq R$%
\[
\phi(\rho)\leq c\left[  \left(  \frac{\rho}{r}\right)  ^{\gamma}%
\phi(r)+Br^{\beta}\right]
\]
where $c$ is a positive constant depending on $A,\alpha,\beta,\gamma.$ \ In
particular, we have for any $0<r\leq R$%
\[
\phi(r)\leq c\left[  \frac{\phi(R)}{R^{\gamma}}r^{\gamma}+Br^{\beta}\right]
.
\]

\end{lemma}

\section{Proofs of the propositions}

We begin by proving Proposition \ref{prop3}.

\begin{proof}
By approximation, (\ref{Cequation}) holds holds for $\eta\in W_{0}^{2,2}.$ We
are assuming that $u\in W^{2,\infty}$, so (\ref{Cequation}) must hold for the
test function%
\[
\eta=-[\tau^{4}u^{h_{p}}]^{-h_{p}}%
\]

where $\tau\in C_{c}^{\infty}$ is a cutoff function in $B_{1}$ that is $1$
on~$B_{1/2}$, and the subscript $h_{p}$ refers to taking difference quotient
in the $e_{p}$ direction. We choose $h$ small enough after having fixed~$\tau
$, so that~$\eta$ is well defined. We have%
\[
\int_{B_{1}}(c^{ij,kl}u_{ij}+f^{kl})[\tau^{4}u_{\,}^{h_{p}}]_{kl}^{-h_{p}%
}dx=0
\]
For $h$ small we can integrate by parts with respect to the difference
quotient to get$\,$%
\[
\int_{B_{1}}(c^{ij,kl}u_{ij}+f^{kl})^{h_{p}}[\tau^{4}u_{\,}^{h_{p}}%
]_{kl}dx=0.
\]

Using the product rule for difference quotients we get%
\[
\int_{B_{1}}[(c^{ij,kl}(x))^{h_{p}}u_{ij}(x)+c^{ij,kl}(x+he_{p})u_{ij}^{h_{p}%
}+(f^{kl})^{h_{p}}][\tau^{4}u_{\,}^{h_{p}}]_{kl}dx=0
\]
Letting $v=u^{h_{p}},$ differentiating the second factor gives:%
\begin{multline*}
\int_{B_{1}}\left[  (c^{ij,kl}(x))^{h_{p}}u_{ij}(x)+c^{ij,kl}(x+he_{p}%
)v_{ij}(x)+(f^{kl})^{h_{p}}(x)\right] \\
\times\left[
\begin{array}
[c]{c}%
\tau^{4}v_{kl}+4\tau^{3}\tau_{k}v_{l}+4\tau^{3}\tau_{l}v_{k}\\
+4v_{\,}(\tau^{3}\tau_{kl}+3\tau^{2}\tau_{k}\tau_{l})
\end{array}
\right]  (x)dx=0
\end{multline*}
from which ~~
\begin{align}
\int_{B_{1}}\tau^{4}c^{ij,kl}(x+he_{p})v_{ij}v_{kl}dx  &  =\nonumber\\
&  -\int_{B_{1}}\left[  (c^{ij,kl}(x))^{h_{p}}u_{ij}(x)+c^{ij,kl}%
(x+he_{p})v_{ij}(x)+(f^{kl})^{h_{p}}(x)\right] \nonumber\\
&  \times\left[
\begin{array}
[c]{c}%
4\tau^{3}\tau_{k}v_{l}+4\tau^{3}\tau_{l}v_{k}\\
+4v_{\,}(\tau^{3}\tau_{kl}+3\tau^{2}\tau_{k}\tau_{l})
\end{array}
\right]  dx\label{star}\\
&  -\int_{B_{1}}\left[  (c^{ij,kl}(x))^{h_{p}}u_{ij}(x)+(f^{kl})^{h_{p}%
}(x)\right]  \tau^{4}v_{kl}dx\nonumber
\end{align}

First we bound the terms on the right side of (\ref{star}). Starting at the
top:\
\begin{align}
&  \int_{B_{1}}\left[  (c^{ij,kl}(x))^{h_{p}}u_{ij}(x)+(f^{kl})^{h_{p}%
}(x)\right]  \times\left[
\begin{array}
[c]{c}%
4\tau^{3}\tau_{k}v_{l}+4\tau^{3}\tau_{l}v_{k}\\
+4v_{\,}(\tau^{3}\tau_{kl}+3\tau^{2}\tau_{k}\tau_{l})
\end{array}
\right]  dx\nonumber\\
&  _{\leq}\left[  \left\Vert u\right\Vert _{W^{2,\infty}(B_{1})}^{2}+1\right]
\int_{B_{1}}\left(  \left\vert (c^{ij,kl}(x))^{h_{p}}\right\vert
^{2}+\left\vert (f^{kl})^{h_{p}}(x)\right\vert ^{2}\right)  dx\label{star0}\\
&  +C_{8}(\tau,D\tau,D^{2}\tau)\int_{B_{1}}\left(  |Dv|^{2}+|v|^{2}\right)
dx.\nonumber
\end{align}

Next, by Young's inequality we have:
\begin{align}
&  \int_{B_{1}}c^{ij,kl}(x+he_{p})v_{ij}(x)\times\nonumber\\
&  \lbrack4\tau^{3}\tau_{j}v_{l}+4\tau^{3}\tau_{l}v_{j}+4v_{\,}(\tau^{3}%
\tau_{jl}+3\tau^{2}\tau_{j}\tau_{l})]dx\nonumber\\
&  \leq\frac{C_{9}(\tau,D\tau,D^{2}\tau,c^{ij,kl})}{\varepsilon}\int_{B_{1}%
}\left(  |Dv|^{2}+v^{2}\right)  dx+\varepsilon\int_{B_{1}}\tau^{4}\left\vert
D^{2}v\right\vert ^{2}dx \label{tryagain}%
\end{align}
and also%
\begin{align}
&  \int_{B_{1}}\left[  (c^{ij,kl}(x))^{h_{p}}u_{ij}(x)+(f^{kl})^{h_{p}%
}(x)\right]  \tau^{4}v_{kl}dx\nonumber\\
&  \leq\varepsilon\int_{B_{1}}\tau^{4}\left\Vert D^{2}v\right\Vert
^{2}dx\nonumber\\
&  +\frac{C_{10}}{\varepsilon}(||u||_{W^{2,\infty}(B_{1})}^{2},|\tau
|_{L^{\infty}(B_{1})})\int_{B_{1}}[|(c^{ijkl})^{h_{p}}|^{2}+|(h^{jl})^{h_{p}%
}|^{2}]dx \label{star3}%
\end{align}

Now by uniform ellipticity (\ref{LH}), the left hand side of (\ref{star}) is
bounded below by
\begin{equation}
\Lambda\int_{B_{1}}\tau^{4}\left\Vert D^{2}v\right\Vert ^{2}dx\leq\int_{B_{1}%
}\tau^{4}c^{ij,kl}(x+he_{p})v_{ik}(x)v_{kl}(x)dx \label{star2}%
\end{equation}
Combining all (\ref{star}), (\ref{star0}) ,(\ref{star3}) , (\ref{tryagain})
and (\ref{star2}) and choosing $\varepsilon$ appropriately, we get
\begin{align*}
&  \frac{\Lambda}{2}\int_{B_{1}}\tau^{4}\left\Vert D^{2}v\right\Vert ^{2}dx\\
&  \leq C_{11}(||\tau||_{W^{2,\infty}(B_{1})},|||u||_{W^{2,\infty}(B_{1})}%
^{2})(\int_{B_{1}}|(f^{kl})^{h_{p}}|^{2}+|c^{ij,kl}|^{2}+|(c^{ij,kl})^{h_{p}%
}|^{2})\\
&  \leq C_{12}(||\tau||_{W^{2,\infty}(B_{1})},||u||_{W^{2,\infty}(B_{1})}%
^{2},||f^{kl}||_{W^{1,2}(B_{1})}^{2},\left\Vert c^{ij,kl}\right\Vert
_{W^{1,2}(B_{1})}^{2},\Lambda).
\end{align*}

Now this estimate is uniform in $h$ and direction $e_{p}$ so we conclude that
the difference quotients of $u$ are uniformly bounded in $W^{2,2}(B_{1/2})$.
Hence $u\in W^{3,2}(B_{1/2})$ and%

\begin{align*}
&  ||D^{3}f||_{L^{2}(B_{1/2})}\\
&  \leq\frac{2C_{12}}{\Lambda}(||\tau||_{W^{2,\infty}(B_{1})}%
,||u||_{W^{2,\infty}(B_{1})}^{2},||f^{kl}||_{W^{1,2}(B_{1})}^{2},\left\Vert
c^{ij,kl}\right\Vert _{W^{1,2}(B_{1})}^{2},\Lambda).
\end{align*}

\end{proof}

We now prove Proposition \ref{prop4}

\begin{proof}
We begin by taking a difference quotient of the equation
\[
\int(c^{ij,kl}u_{ij}+f^{kl})\eta_{kl}dx=0
\]
along the direction $h_{m}$ . This gives
\[
\int[(c^{ij,kl}(x))^{h_{m}}u_{ij}(x)+c^{ij,kl}(x+he_{m})u_{ij}^{h_{m}%
}(x)+(f^{kl})^{h_{m}}]\eta_{kl}(x)dx=0
\]
which gives us the following PDE in $u_{ij}^{h_{m}}:$
\[
\int c^{ij,kl}(x+he_{m})u_{ij}^{h_{m}}(x)\eta_{kl}(x)dx=\int q(x)\eta
_{kl}(x)dx
\]

where
\[
q(x)=-(f^{kl})^{h_{m}}(x)-(c^{ij,kl}(x))^{h_{m}}u_{ij}(x)
\]
Note that~$q\in C^{\alpha}(B_{1})$ and~$c^{ij,kl}(x+he_{m})$ is still an
elliptic term for all $x$ in~$B_{1.}$ For compactness of notation we denote
\begin{equation}
g=u^{h_{m}} \label{defineg}%
\end{equation}
and replace \ $c^{ij,kl}(x+he_{m})$ with $c^{ij,kl},$ as the difference is
immaterial. \ Our equation reduces to%
\begin{equation}
\int c^{ij,kl}g_{ij}\eta_{kl}dx=\int q\eta_{kl}dx \label{eq787}%
\end{equation}
Using integration by parts we have%
\begin{align*}
\int c^{ij,kl}g_{ij}\eta_{kl}dx  &  =-\int q_{l}\eta_{k}dx\\
&  =-\int(q-q(0))_{l}\eta_{k}dx\\
&  =\int(q-q(0))\eta_{kl}dx
\end{align*}

Now for each fixed $r<1$ we write $g=v+w$ where $w$ satisfies the following
constant coefficient PDE on~$B_{r}\subseteq B_{1}:$%
\begin{align}
\int_{B_{1}(0)}c^{ij,kl}(0)w_{ij}\eta_{kl}dx  &  =0\label{testv}\\
\forall\eta &  \in C_{0}^{\infty}(B_{r}(0))\nonumber\\
w  &  =g\text{ \ on}~\partial B_{r}\nonumber\\
\nabla w  &  =\nabla g\text{\ on}~\partial B_{r}.\nonumber
\end{align}
~ By the Lax Milgram Theorem the above PDE with the given boundary condition
has a unique solution. By combining (\ref{eq787}) and (\ref{testv}) we
conclude
\begin{equation}
\int_{B_{r}}c^{ij,kl}(0)v_{ij}\eta_{kl}dx=\int_{B_{r}}(c^{ij,kl}%
(0)-c^{ij,kl}(x))g_{ij}\eta_{kl}dx+\int_{B_{r}}q\eta_{kl}dx \label{testv2}%
\end{equation}

Now $w$ is smooth (again see \cite[Theorem 33.10]{Driver03}), and $g=u^{h_{m}%
}$ is $C^{2,\alpha},$ so $v=g-w$ is $C^{2,\alpha}$ and can be well
approximated by smooth test functions in $H_{0}^{2}(B_{r}).$ \ It follows that
$v$ can be used as a test function in (\ref{testv2}):\ On the left hand side
we have by (\ref{LH})
\[
\left[  \int_{B_{r}}c^{ij,kl}(0)v_{ij}v_{kl}dx\right]  ^{2}\geq\left[
\Lambda\int_{B_{r}}|D^{2}v|^{2}dx\right]  ^{2}.
\]
Defining
\begin{equation}
\zeta(r)=\sup\{\mid c^{ij,kl}(x)-c^{ij,kl}(y)|:x,y\in B_{r}\} \label{ccalpha}%
\end{equation}
and using the Cauchy-Schwarz inequality we get
\[
\left[  \int_{B_{r}}(c^{ij,kl}(0)-c^{ij,kl}(x))g_{ij}v_{kl}dx\right]  ^{2}%
\leq\zeta^{2}(r)\int_{B_{r}}|D^{2}g|^{2}dx\int_{B_{r}}|D^{2}v|^{2}dx.
\]

Using Holder's inequality%
\[
\left[  \int_{B_{r}}\left\vert (q(x)-q(0))v_{kl}(x)\right\vert dx\right]
^{2}\leq\int_{B_{r}}|q(x)-q(0)|^{2}dx\int_{B_{r}}|D^{2}v|^{2}dx
\]
This gives us
\[
\Lambda^{2}\left[  \int_{B_{r}}|D^{2}v|^{2}dx\right]  ^{2}\leq\zeta^{2}%
(r)\int_{B_{r}}|D^{2}g|^{2}dx\int_{B_{r}}|D^{2}v|^{2}dx+\int_{B_{r}%
}|q(x)-q(0)|^{2}dx\int_{B_{r}}|D^{2}v|^{2}dx
\]
which implies%
\begin{equation}
\Lambda^{2}\int_{B_{r}}|D^{2}v|^{2}dx\leq\zeta^{2}(r)\int_{B_{r}}|D^{2}%
g|^{2}dx+\int_{B_{r}}|q(x)-q(0)|^{2}dx. \label{eq799}%
\end{equation}
Using corollary \ref{Cor2} \ for any $0<\rho\leq r$~ we get%
\begin{equation}
\int_{B_{\rho}}\left\vert D^{2}g\right\vert ^{2}dx\leq4C_{1}(\rho
/r)^{n}\left\Vert D^{2}g\right\Vert _{L^{2}(B_{r})}^{2}+\left(  2+8C_{1}%
\right)  \left\Vert D^{2}v\right\Vert _{L^{2}(B_{r})}^{2} \label{eq800}%
\end{equation}
Now combing (\ref{eq800}) and (\ref{eq799}) we get
\begin{align}
\int_{B_{\rho}}\left\vert D^{2}g\right\vert ^{2}dx  &  \leq4C_{1}(\rho
/r)^{n}\left\Vert D^{2}g\right\Vert _{L^{2}(B_{r})}^{2}\nonumber\\
&  +\frac{\left(  2+8C_{1}\right)  }{\Lambda^{2}}\left[  \zeta^{2}%
(r)\int_{B_{r}}|D^{2}g|^{2}dx+\int_{B_{r}}|q(x)-q(0)|^{2}dx\right] \nonumber\\
&  =\left[  \frac{\left(  2+8C_{1}\right)  \zeta^{2}(r)}{\Lambda^{2}}%
+4C_{1}(\rho/r)^{n}\right]  \int_{B_{r}}|D^{2}g|^{2}dx\nonumber\\
&  +\frac{\left(  2+8C_{1}\right)  }{\Lambda^{2}}\int_{B_{r}}|q(x)-q(0)|^{2}%
dx. \label{A0}%
\end{align}
Also from Corollary \ref{Cor2} \
\begin{align*}
\int_{B_{\rho}}\left\vert D^{2}g-(D^{2}g)_{\rho}\right\vert ^{2}dx  &
\leq4C_{2}(\rho/r)^{n+2}\int_{B_{r}}\left\vert D^{2}g-(D^{2}g)_{r}\right\vert
^{2}dx\\
&  +\left(  8+16C_{2}\right)  \int_{B_{r}}\left\vert D^{2}v\right\vert
^{2}dx\\
&  \leq4C_{2}(\rho/r)^{n+2}\int_{B_{r}}\left\vert D^{2}g-(D^{2}g)_{\rho
}\right\vert ^{2}dx\\
&  +\frac{\left(  8+16C_{2}\right)  }{\Lambda^{2}}\left[  \zeta^{2}%
(r)\int_{B_{r}}|D^{2}g|^{2}dx+\int_{B_{r}}|q(x)-q(0)|^{2}dx\right]  .
\end{align*}

Because $c^{ij,kl}\in C^{1,\alpha}$ we have from (\ref{ccalpha}) that
\begin{equation}
\zeta(r)^{2}\leq C_{13}r^{2\alpha}%
\end{equation}
Again $q$ is a~$C^{\alpha}$ function which implies%
\[
\left\vert q(x)-q(0)\right\vert \leq\left\Vert q\right\Vert _{C^{\alpha}%
(B_{1})}|x-0|^{\alpha}%
\]
and
\[
\int_{B_{r}}|q-q(0)|^{2}dx\leq C_{14}\left\Vert q\right\Vert _{C^{\alpha
}(B_{1})}r^{n+2\alpha}%
\]
So we have%
\begin{align}
&  \int_{B_{\rho}}|D^{2}g-(D^{2}g)_{\rho}|^{2}\label{A1}\\
&  \leq4C_{2}(\rho/r)^{n+2}\int_{B_{r}}\left\vert D^{2}g-(D^{2}g)_{\rho
}\right\vert ^{2}\nonumber\\
&  +\frac{\left(  8+16C_{2}\right)  }{\Lambda^{2}}C_{13}r^{2\alpha}\int%
_{B_{r}}|D^{2}g|^{2}\nonumber\\
&  +\frac{\left(  8+16C_{2}\right)  }{\Lambda^{2}}C_{14}\left\Vert
q\right\Vert _{C^{\alpha}(B_{1})}r^{n+2\alpha}.\nonumber
\end{align}

For $r<r_{0}<1/4$ to be determined, we have (\ref{A0})
\[
\int_{B_{\rho}}\left\vert D^{2}g\right\vert ^{2}\leq C_{15}\left\{
[(\rho/r)^{n}+r^{2\alpha}]\int_{B_{r}}\left\vert D^{2}g\right\vert ^{2}%
+r_{0}^{2\alpha+2\delta}r^{n-2\delta}\right\}  .
\]
Where $\delta$ is some positive number. \ Now we apply \cite[Lemma 3.4]{HL}.
\ In particular, take
\begin{align*}
\phi(\rho)  &  =\int_{B_{\rho}}\left\vert D^{2}g\right\vert ^{2}\\
A  &  =C_{15}\\
B  &  =r_{0}^{2\alpha+2\delta}\\
\alpha &  =n\\
\beta &  =n-2\delta\\
\gamma &  =n-\delta.
\end{align*}
There exists $\varepsilon_{0}(A,\alpha,\beta,\gamma)$ such that if
\begin{equation}
r_{0}^{2\alpha}\leq\varepsilon_{0} \label{rnot}%
\end{equation}
we have
\[
\phi(\rho)\leq C_{15}\left\{  [(\rho/r)^{n}+\varepsilon_{0}]\phi
(r)+r_{0}^{2\alpha+2\delta}r^{n-2\delta}\right\}
\]
and the conclusion of \cite[Lemma 3.4]{HL} says that for $\rho<r_{0}$%
\begin{align*}
\phi(\rho)  &  \leq C_{16}\left\{  [(\rho/r)^{\gamma}]\phi(r)+r_{0}%
^{2\alpha+2\delta}\rho^{n-2\delta}\right\} \\
&  \leq C_{16}\frac{1}{r_{0}^{n-\delta}}\rho^{n-\delta}\left\Vert
D^{2}g\right\Vert _{L^{2}(B_{r_{0}})}+r_{0}^{2\alpha+2\delta}\rho^{n-2\delta
}\\
&  \leq C_{17}\rho^{n-\delta}%
\end{align*}

This $C_{17}$ depends on $r_{0}$ which is chosen by (\ref{rnot}) and
$\left\Vert D^{2}g\right\Vert _{L^{2}(B_{3/4})}$. So there is a positive
uniform radius upon which this holds for points well in the interior. In
particular, we choose $r_{0}\in(0,1/4)$ so that the estimate can be applied
uniformly at points centered in $B_{1/2}(0)$ whose balls remain in
$B_{3/4}(0)$. Turning back to (\ref{A1}), we now have,%
\begin{align*}
\int_{B_{\rho}}|D^{2}g-(D^{2}g)_{\rho}|^{2}  &  \leq4C_{2}(\rho/r)^{n+2}%
\int_{B_{r}}\left\vert D^{2}g-(D^{2}g)_{\rho}\right\vert ^{2}+C_{18}%
r^{2\alpha}\rho^{n-\delta}\\
&  +C_{19}\left\Vert q\right\Vert _{C^{\alpha}(B_{1})}r^{n+2\alpha}\\
&  \leq4C_{2}(\rho/r)^{n+2}\int_{B_{r}}\left\vert D^{2}g-(D^{2}g)_{\rho
}\right\vert ^{2}+C_{20}r^{n+2\alpha-\delta}%
\end{align*}
Again we apply \cite[Lemma 3.4]{HL}: This time, take
\begin{align*}
\phi(\rho)  &  =\int_{B_{\rho}}|D^{2}g-(D^{2}g)_{\rho}|^{2}\\
A  &  =4C_{2}\\
B  &  =C_{20}\\
\alpha &  =n+2\\
\beta &  =n+2\alpha-\delta\\
\gamma &  =n+2\alpha
\end{align*}
and conclude that for any $r<r_{0}$%
\begin{align*}
\int_{B_{r}}|D^{2}g-(D^{2}g)_{\rho}|^{2}  &  \leq C_{21}\left\{  \frac
{1}{r_{0}^{n+2\alpha}}\int_{B_{r_{0}}}|D^{2}g-(D^{2}g)_{r_{0}}|^{2}%
r^{n+2\alpha}+C_{20}r^{n+2\alpha-\delta}\right\} \\
&  \leq C_{22}r^{n+2\alpha-\delta}%
\end{align*}
with $C_{22}$ depending on $r_{0},\left\Vert D^{2}g\right\Vert _{L^{2}%
(B_{3/4})}$, $\left\Vert q\right\Vert _{C^{\alpha}(B_{1})}$ etc. \ It follows
by \cite[Theorem 3.1]{HL} that $D^{2}g\in C^{\left(  2\alpha-\delta\right)
/2}(B_{1/4}),$ in particular, must be bounded locally:%
\begin{equation}
\left\Vert D^{2}g\right\Vert _{L^{\infty}(B_{1/4})}\leq C_{23}\left\{
1+\left\Vert D^{2}g\right\Vert _{L^{2}(B_{1/2})}\right\}  .
\label{repeatlater}%
\end{equation}
This allows us to bound
\[
\int_{B_{r}}|D^{2}g|^{2}\leq C_{24}r^{n}%
\]
which we can plug back in to (\ref{A1}):%
\begin{align*}
\int_{B_{\rho}}|D^{2}g-(D^{2}g)_{\rho}|^{2}  &  \leq4C_{2}(\rho/r)^{n+2}%
\int_{B_{r}}\left\vert D^{2}g-(D^{2}g)_{\rho}\right\vert ^{2}+C_{25}%
r^{2\alpha}C_{24}r^{n}\\
&  +C_{19}\left\Vert q\right\Vert _{C^{\alpha}(B_{1})}r^{n+2\alpha}\\
&  \leq C_{26}r^{n+2\alpha}%
\end{align*}

This is precisely the hypothesis in \cite[Theorem 3.1]{HL}. \ We conclude
that
\[
\left\Vert D^{2}g\right\Vert _{C^{\alpha}(B_{1/4})}\leq C_{27}\left\{
\sqrt{C_{26}}+\left\Vert D^{2}g\right\Vert _{L^{2}(B_{1/2})}\right\}  .
\]
Recalling (\ref{defineg}) we see that $u$ must enjoy uniform $C^{3,\alpha}$
estimates on the interior, and the result follows.
\end{proof}

\section{Proof of the Theorem}

The propositions in the previous section allow us to prove the following
Corollary, from which the Main Theorem will follow. \ 

\begin{corollary}
Suppose~$u\in C^{N,\alpha}(B_{1})$ , $N\geq2,$and satisfies the following
regular (recall (\ref{Bdef})) fourth order equation
\[
\int_{\Omega}a^{ij,kl}(D^{2}u)u_{ij}\eta_{kl}dx=0,\text{ }\forall\eta\in
C_{0}^{\infty}(\Omega).
\]
Then
\[
\left\Vert u\right\Vert _{C^{N+1,\alpha}(B_{r})}\leq C(n,b,\left\Vert
u\right\Vert _{W^{N,\infty}(B_{1})}).
\]
In particular
\[
u\in C^{N,\alpha}(B_{1})\implies u\in C^{N+1,\alpha}(B_{r})
\]

\end{corollary}

\textbf{Case 1} $N=2.$ The function $u\in C^{2,\alpha}\left(
B_{1}\right)  $ and hence also in $W^{2,\infty}\left(  B_{1}\right)  $ . By
approximation (\ref{eq1}) holds for $\eta\in W_{0}^{2,\infty},$ in particular,
for%
\[
\eta=-[\tau^{4}u^{h_{m}}]^{-h_{m}}%
\]
where $\tau\in C_{c}^{\infty}\left(  B_{1}\right)  $ is a cut off function in
$B_{1}$ that is $1$ on~$B_{1/2}$, and superscript $h_{m}$ refers to the
difference quotient. As before, we have chosen $h$ small enough (depending
on~$\tau$) so that~$\eta$ is well defined . We have%
\[
\int_{\Omega}a^{ij,kl}(D^{2}u)u_{ij}\left[  \tau^{4}f^{h_{m}}\right]
_{kl}dx=0.
\]
Integrating by parts as before with respect to the difference quotient, we get%
\[
\int_{B_{1}}[a^{ij,kl}(D^{2}f)u_{ij}]^{h_{m}}[\tau^{4}u^{h_{m}}]_{kl}dx=0
\]
Let $v=u^{h_{m}}$. Observe that the first difference quotient can be
expressed as
\begin{align}
\lbrack a^{ij,kl}(D^{2}f)u_{ij}]^{h_{m}}(x)  &  =a^{ij,kl}(D^{2}%
u(x+he_{m}))\frac{u_{ij}(x+he_{m})-u_{ij}(x)}{h}\label{diff_of_a}\\
&  +\frac{1}{h}\left[  a^{ij,kl}(D^{2}u(x+he_{m}))-a^{ij,kl}(D^{2}%
u(x))\right]  u_{ij}(x)\nonumber\\
&  =a^{ij,kl}(D^{2}u(x+he_{m}))v_{ij}(x)\nonumber\\
&  +\left[  \int_{0}^{1}\frac{\partial a^{ij,kl}}{\partial u_{pq}}%
(tD^{2}u(x+he_{m})+(1-t)D^{2}u(x))dt\right]  v_{pq}(x)u_{ij}(x).\nonumber
\end{align}
We get%
\begin{equation}
\int_{B_{1}}\tilde{b}^{ij,kl}v_{ij}[\tau^{4}v]_{kl}dx=0 \label{dq3}%
\end{equation}
where%
\begin{equation}
\tilde{b}^{ij,kl}(x)=a^{ij,kl}(D^{2}u(x+he_{m}))+\left[  \int_{0}^{1}%
\frac{\partial a^{pq,kl}}{\partial u_{ij}}(tD^{2}u(x+he_{m})+(1-t)D^{2}%
u(x))dt\right]  u_{pq}(x). \label{btwid}%
\end{equation}
Expanding derivatives of the second factor\ in (\ref{dq3}) and collecting
terms gives us
\[
\int_{B_{1}}\tilde{b}^{ij,kl}v_{ij}\tau^{4}v_{kl}dx\leq\int_{B_{1}}\left\vert
\tilde{b}^{ij,kl}\right\vert \left\vert v_{ij}\right\vert \tau^{2}C_{28}%
(\tau,D\tau,D^{2}\tau)\left(  1+|v|+|Dv|\right)  dx\,
\]
Now for $h$ small,   $\tilde{b}^{ij,kl}$ very closely approximates  $b^{ij,kl},$
so we may assume $h$ is small. Applying (\ref{Bcondition})) and Young's
inequality%
\[
\int_{B_{1}}\tau^{4}\Lambda_{1}|D^{2}v|^{2}\leq C_{28}\sup\tilde{b}%
^{ij,kl}\int_{B_{1}}\left(  \varepsilon\tau^{4}|D^{2}v|^{2}+C_{32}\frac
{1}{\varepsilon}(1+|v|+|Dv|)^{2}\right)  dx.
\]
That is
\[
\int_{B_{1/2}}|D^{2}v|^{2}\leq C_{29}\int_{B_{1}}(1+|v|+|Dv|)^{2}dx.
\]
Now this estimate is uniform in $h$ (for $h$ small enough) and direction
$e_{m,}$ so we conclude that the derivatives are in $W^{2,2}(B_{1/2}).$ This
also shows that%
\[
||D^{3}u||_{L^{2}(B_{1/2})}\leq C_{30}\left(  ||Du||_{L^{2}(B_{1})},\left\Vert
D^{2}u\right\Vert _{L^{2}(B_{1})}\right)  .
\]
Remark : We only used uniform continuity of $D^{2}u$ to allow us to take the
limit, but we did require the precise modulus of continuity.

For the next step, we are not quite able to use Proposition \ref{prop4}
because the coefficients $a^{ij,kl}$ are only known to be $W^{1,2}$. \ So we
proceed by hand. \ \ 

We begin by taking a single difference quotient%
\[
\int_{B_{1}}[a^{ij,kl}(D^{2}f)u_{ij}]^{h_{m}}\eta_{kl}dx=0
\]
and arriving at the equation in the same fashion as to (\ref{dq3}) above (this
time letting $g=u^{h_{m}}$) we have
\[
\int_{B_{1}}\tilde{b}^{ij,kl}g_{ij}(x)\eta_{kl}dx=0.
\]
Inspecting (\ref{btwid}) we see that $\tilde{b}^{ij,kl}$ is $C^{\alpha}:$ \
\[
\left\Vert \tilde{b}^{ij,kl}(x)-\tilde{b}^{ij,kl}(y)\right\Vert \leq
C_{31}\left\vert x-y\right\vert ^{\alpha}\text{ }%
\]
where $C_{31}$ depends on $\left\Vert D^{2}u\right\Vert _{C^{\alpha}}$ and on
bounds of $Da^{ij,kl}$ and $D^{2}a^{ij,kl}.$\ As in the proof of Proposition
\ref{prop4}, \ for a fixed $r\,<1$ \ we let $w$ solve the boundary value
problem
\begin{align*}
\int\tilde{b}^{ij,kl}(0)w_{ij}\eta_{kl}dx  &  =0,\forall\eta\in C_{0}^{\infty
}(B_{r})\\
w  &  =g\text{ on }\partial B_{r}\\
\nabla w  &  =\nabla g\text{ on }\partial B_{r}%
\end{align*}
Let $v=g-w.$ Note that%
\[
\int\tilde{b}^{ij,kl}(0)v_{ij}\eta_{kl}dx=\int\left(  \tilde{b}^{ij,kl}%
(0)-\tilde{b}^{ij,kl}(x)\right)  g_{ij}\eta_{kl}dx.
\]
Now $v$ vanishes to second order on the boundary, and we may use $v$ as a test
function. We get
\[
\int\tilde{b}^{ij,kl}(0)v_{ij}v_{kl}dx=\int\left(  \tilde{b}^{ij,kl}%
(0)-\tilde{b}^{ij,kl}(x)\right)  g_{ij}v_{kl}dx.
\]
As before,
\[
\left(  \Lambda\int_{B_{r}}\left\vert D^{2}v\right\vert ^{2}dx\right)
^{2}\leq\left[  \sup_{x\in B_{r}}\left\vert \tilde{b}^{ij,kl}(0)-\tilde
{b}^{ij,kl}(x)\right\vert \right]  ^{2}\int_{B_{r}}\left\vert D^{2}%
g\right\vert ^{2}dx\int_{B_{r}}\left\vert D^{2}v\right\vert ^{2}dx.
\]

Defining
\begin{align}
\zeta(r)  &  =\sup\{\left\vert \tilde{b}^{ij,kl}(x)-\tilde{b}^{ij,kl}%
(y)\right\vert x,y\in B_{r}\}\label{alphab}\\
&  \leq4^{\alpha}C_{31}r^{2\alpha}\nonumber
\end{align}
then%
\[
\int_{B_{r}}(\tilde{b}^{ij,kl}(0)-\tilde{b}^{ij,kl}(x))g_{ij}v_{kl}dx)^{2}%
\leq\zeta^{2}(r)\int_{B_{r}}\left\vert D^{2}g\right\vert ^{2}\int_{B_{r}%
}\left\vert D^{2}v\right\vert ^{2}.
\]
So now we have :%

\[
\int_{B_{r}}\left\vert D^{2}v\right\vert ^{2}\leq\frac{\zeta^{2}(r)}%
{\Lambda^{2}}\int_{B_{r}}\left\vert D^{2}g\right\vert ^{2}.
\]
Using Corollary \ref{Cor2}, for any $0<\rho\leq r$ we get%

\begin{align}
\int_{B_{\rho}}\left\vert D^{2}g-(D^{2}g)_{\rho}\right\vert ^{2}  &
\leq4C_{2}(\rho/r)^{n+2}\int_{B_{r}}\left\vert D^{2}g-(D^{2}g)_{r}\right\vert
^{2}\nonumber\\
&  +\left(  8+16C_{2}\right)  \int_{B_{r}}\left\vert D^{2}v\right\vert
^{2}\nonumber\\
&  \leq4C_{2}(\rho/r)^{n+2}\int_{B_{r}}\left\vert D^{2}g-(D^{2}g)_{r}%
\right\vert ^{2}+\frac{\left(  8+16C_{2}\right)  \zeta^{2}(r)}{\Lambda^{2}%
}\left\Vert D^{2}g\right\Vert _{L^{2}(B_{r})}^{2}. \label{fromcor6}%
\end{align}
Also by Corollary \ref{Cor2}%

\begin{align*}
\int_{B_{\rho}}\left\vert D^{2}g\right\vert ^{2}  &  \leq4C_{1}(\rho
/r)^{n}\left\Vert D^{2}g\right\Vert _{L^{2}(B_{r})}^{2}+\left(  2+8C_{1}%
\right)  \left\Vert D^{2}v\right\Vert _{L^{2}(B_{r})}^{2}\\
&  \leq4C_{1}(\rho/r)^{n}\left\Vert D^{2}g\right\Vert _{L^{2}(B_{r})}%
^{2}+\left(  2+8C_{1}\right)  \frac{\zeta^{2}(r)}{\Lambda^{2}}\left\Vert
D^{2}g\right\Vert _{L^{2}(B_{r})}^{2}.
\end{align*}
This implies%

\[
\int_{B_{\rho}}\left\vert D^{2}g\right\vert ^{2}\leq\left(  4C_{1}(\rho
/r)^{n}+\left(  2+8C_{1}\right)  4^{2\alpha}C_{31}^{2}r^{2\alpha}\right)
\left\Vert D^{2}g\right\Vert _{L^{2}(B_{r})}^{2}.
\]
Now we can apply \cite[Lemma 3.4]{HL} again, this time with
\begin{align*}
\phi(\rho)  &  =\int_{B_{\rho}}\left\vert D^{2}g\right\vert ^{2}\\
A  &  =4C_{1}\\
\alpha &  =n\\
B,\beta &  =0\\
\gamma &  =n-2\delta\\
\varepsilon &  =\left(  2+8C_{1}\right)  4^{2\alpha}C_{31}^{2}r^{2\alpha}.
\end{align*}
There exists a constant $\varepsilon_{0}(A,\alpha,\gamma)$ such that by
chosing
\[
r_{0}^{2\alpha}\leq\frac{\varepsilon_{0}}{\left(  2+8C_{1}\right)  4^{2\alpha
}C_{31}^{2}}<\frac{1}{4}%
\]
we may conclude that for $0<r\leq r_{0}$%
\begin{equation}
\int_{B_{r}}\left\vert D^{2}g\right\vert ^{2}\leq C_{32}r^{n-2\delta}%
\frac{\int_{Br_{0}}\left\vert D^{2}g\right\vert ^{2}}{r_{0}^{n-2\delta}}.
\label{conclusion5}%
\end{equation}

Next, for small $\rho<r<r_{0}$ we have combining (\ref{fromcor6})
(\ref{alphab}) and (\ref{conclusion5})
\begin{align}
\int_{B\rho}\left\vert D^{2}g-(D^{2}g)_{\rho}\right\vert ^{2}  &  \leq
4C_{2}(\rho/r)^{n+2}\int_{B_{r}}\left\vert D^{2}g-(D^{2}g)_{r}\right\vert
^{2}\label{bigintegral}\\
&  +\frac{\left(  8+16C_{2}\right)  4^{\alpha}}{\Lambda^{2}}\frac{\int%
_{Br_{0}}\left\vert D^{2}g\right\vert ^{2}}{r_{0}^{n-2\delta}}C_{31}%
C_{32}r^{n-2\delta}r^{2\alpha}\nonumber\\
&  \leq C_{33}r^{n+2\alpha-\delta}\nonumber
\end{align}
with $C_{33}$ depending on $\left\Vert D^{2}g\right\Vert _{L^{2}(B_{3/4}%
)},r_{0},\varepsilon_{0}$. Again, we apply \cite[Theorem 3.1]{HL} \ to
$D^{2}g\in C^{\left(  2\alpha-\delta\right)  /2}(B_{1/4}).$ \ From here, the
argument is identical to the argument following (\ref{repeatlater}). We
conclude that%
\[
\left\Vert D^{2}g\right\Vert _{C^{\alpha}(B_{1/4})}\leq C_{34}\left\{
1+\left\Vert D^{2}g\right\Vert _{L^{2}(B_{3/4})}\right\}  .
\]
Substituting $g=u^{h_{m}}$ we see that $u$ must enjoy uniform $C^{3,\alpha}$
estimates on the interior, and the result follows.

\textbf{Case 2 }$N=3.$ \ We may take a difference quotient\ of (\ref{eq1})
directly.%
\[
\int_{\Omega}\left[  a^{ij,kl}(D^{2}u)u_{ij}\right]  ^{h_{m}}\eta
_{kl}dx=0,\text{ }\forall\eta\in C_{0}^{\infty}(\Omega).
\]
(To be more clear we, are using a slightly offset test function
$\eta(x+he_{m})$ and then using a change of variables, subtracting, and
dividing by $h.$)\ 

We get
\[
\int_{B_{1}}\left[  a^{ij,kl}(D^{2}u(x+he_{m}))u_{ij}^{h_{m}}(x)+\frac
{\partial a^{ij,kl}}{\partial u_{pq}}(M^{\ast}(x))u_{pq}^{h_{m}}%
(x)u_{ij}(x)\right]  \eta_{kl}=0.
\]
where $M^{\ast}(x)=t^{\ast}D^{2}u(x+h_{m})+(1-t^{\ast})D^{2}u(x)$ and
$t^{\ast}\in\lbrack0,1]$. Now we are assuming that $u\in C^{3,\alpha
}(B_{1}),$ so the first and second derivatives of the difference quotient
will converge to the second and third derivatives, uniformly.
We can then apply dominated convergence, passing the limit as
$h\rightarrow0$ inside the integral and recalling $u_{m}=v$ as before, we
obtain \
\[
\int_{B_{1}}\left[  [a^{ij,kl}(D^{2}u(x))v_{ij}(x)+\frac{\partial a^{pq,kl}%
}{\partial u_{ij}}\left(  D^{2}u(x)\right)  v_{ij}(x)u_{pq}(x)\right]
\eta_{kl}=0
\]
that is
\begin{equation}
\int_{B_{1}}b^{ij,kl}(D^{2}u(x))v_{ij}(x)\eta_{kl}(x)=0,\text{ \ \ }%
\forall\eta\in C_{0}^{\infty}(\Omega). \label{eqb3}%
\end{equation}

It follows that $v\in C^{2,\alpha\text{ }}$satisfies a fourth order double
divergence equation, with coefficients in $C^{1,\alpha}.$ \ First, we apply
Proposition \ref{prop3}\ :
\[
\left\Vert D^{3}v\right\Vert _{L^{2}(B_{1/2})}\leq C_{35}\left(
||v||_{W^{2,\infty}(B_{1})}\right)  (1+||b^{ij,kl}||_{W^{1,2}(B_{1})}).
\]
In particular, $u\in W^{4,2}(B_{1/2}).$ Next, we apply \ref{prop4}
\begin{align*}
||D^{3}v||_{C^{0,\alpha}(B_{1/4})}  &  \leq C(1+||D^{3}v||_{L^{2}(B_{1/2}%
)})\leq C(||u||_{W^{2,\infty}(B_{1})},|b^{ij,kl}||_{W^{1,2}(B_{1})})\\
&  \leq C_{36}(n,b,\left\Vert u\right\Vert _{C^{3,\alpha}(B_{1})}).
\end{align*}
We conclude that $u\in C^{4,\alpha}(B_{r})$ for any $r<1.$

\textbf{Case 3 } $N\geq4$. \ \ Let $v=D^{\alpha}u$ for some multindex $\alpha$
with $\left\vert \alpha\right\vert =N-2.$ Observe that taking the first
difference quotient and then taking a limit yields (\ref{eqb3}), when $u\in
C^{3,\alpha}.$ \ Now if $u\in C^{4,\alpha}$ we may take a difference quotient
and limit of (\ref{eqb3}) to obtain%
\[
\int_{B_{1}}\left[  b^{ij,kl}(D^{2}u(x))u_{ijm_{1}m_{2}}(x)+\frac{\partial
b^{ij,kl}}{\partial u_{pq}}(D^{2}u(x))u_{pqm_{2}}u_{ij}\right]  \eta
_{kl}(x)=0,\text{ \ \ }\forall\eta\in C_{0}^{\infty}(\Omega).
\]
and if $u\in C^{N,\alpha}$, then $v\in C^{2,\alpha}$, so we may take $N-2$
difference quotients to obtain%
\begin{equation}
\int_{B_{1}}\left[  b^{ij,kl}(D^{2}u(x))v_{ij}(x)+f^{kl}(x)\right]  \eta
_{kl}(x)=0,\text{ \ \ }\forall\eta\in C_{0}^{\infty}(\Omega).
\label{inductiveeq}%
\end{equation}
where
\[
f^{kl}=D^{\alpha}\left(  b^{ij,kl}(D^{2}u(x))u_{ij}\right)  -b^{ij,kl}%
(D^{2}u(x))D^{\alpha}u_{ij}.
\]
One can check by applying the chain rule repeatedly that $f^{kl}$ is
$C^{1,\alpha}.$ \ So we may apply Proposition \ref{prop3}\ \ to
(\ref{inductiveeq}) and obtain that
\[
\left\Vert D^{3}v\right\Vert _{L^{2}(B_{1/2})}\leq C_{37}(\left\Vert
v\right\Vert _{W^{2,\infty}(B_{1})})(1+||b^{ij,kl}||_{W^{1,2}(B_{1})})
\]
that is
\[
\left\Vert u\right\Vert _{W^{N+1,2}(B_{r})}\leq C_{38}(n,b,\left\Vert
u\right\Vert _{W^{N,\infty}(B_{1})}).
\]
Now apply Proposition \ref{prop4}:%
\[
||D^{3}v||_{C^{0,\alpha}(B_{1/4})}\leq C_{39}(1+||D^{3}v||_{L^{2}(B_{3/4})})
\]
that is
\[
\left\Vert u\right\Vert _{C^{N+1,\alpha}(B_{r})}\leq C_{40}(n,b,\left\Vert
u\right\Vert _{W^{N,\infty}(B_{1})}).
\]

The Main Theorem follows.

\section{Critical Points of Convex Functions of the Hessian}

Suppose that $F(D^{2}u)$ is either a convex or a concave function of $D^{2}u,$
and we have found a critical point of
\begin{equation}
\int_{\Omega}F(D^{2}u)dx\label{Ffunc}%
\end{equation}
for some $\Omega\subset%
\mathbb{R}
^{n},$ where we are restricting to compactly supported variations, so the that
Euler-Lagrange equation is (\ref{generic}). \ \ If we suppose that $F$ also
has the additional structure condition,
\begin{equation}
\frac{\partial F(D^{2}u)}{\partial u_{ij}}=a^{pq,ij}(D^{2}u)u_{pq}%
\label{structc}%
\end{equation}
for a some $a^{ij,kl}$ satisfying (\ref{LH}), then we can derive smoothness
from $C^{2,\alpha},$ as follows.

\begin{corollary}
~Suppose $u\in C^{2,\alpha}(B_{1})$ is critical point of (\ref{Ffunc}), where
$F$ is a smooth function satisfying (\ref{structc}) with $a^{ij,kl}$
satisfying (\ref{LH}) and F is uniformly convex or uniformly concave on
$U\subseteq S^{n\times n}$ where~$U$ is the range of~$D^{2}u(B_{1})$ in the
Hessian space. 

Then $u\in C^{\infty}(B_{r})$, for all $r<1.$

\end{corollary}

\begin{proof}
If $u$ is a critical point of  (\ref{Ffunc}), then it satisfies the weak
equation (\ref{eq1}), for $a^{ij,kl}$ in (\ref{structc}). To apply the
main Theorem, all we need to show is that
\[
b^{ij,kl}(D^{2}u(x))=a^{ij,kl}(D^{2}u(x))+\frac{\partial a^{pq,kl}}{\partial
u_{ij}}(D^{2}u(x))u_{pq}(x)
\]
satisfies (\ref{LH}). \ $\ $From  (\ref{structc}):%
\begin{equation}
\frac{\partial}{\partial u_{kl}}\left(  \frac{\partial F(D^{2}u)}{\partial
u_{ij}}\right)  =a^{kl,ij}(D^{2}u)+\frac{\partial a^{pq,ij}(D^{2}u)}{\partial
u_{kl}}u_{pq}.
\end{equation}
So
\[
b^{ij,kl}(D^{2}u(x))\xi_{ij}\xi_{kl}=\frac{\partial}{\partial u_{kl}}\left(
\frac{\partial F(D^{2}u)}{\partial u_{ij}}\right)  \xi_{ij}\xi_{kl}\geq
\Lambda\left\vert \xi\right\vert ^{2}%
\]
for some $\Lambda>0$, because $F$ is convex. If $F$ is concave, $u$ is
still a critical point of $-F$ and the same argument holds. 
\end{proof}

\bigskip

We mention one special case.

\begin{lemma}
Suppose $F(D^{2}u)=f(w)$ where $w=(D^{2}u)^{T}(D^{2}u).$ \ Then
\begin{equation}
\frac{\partial F(D^{2}u)}{\partial u_{ij}}=a^{ij,kl}(D^{2}u)u_{kl}%
\label{squares}%
\end{equation}

\end{lemma}

\begin{proof}
Let%
\[
w_{kl}=u_{ka}\delta^{ab}u_{bl}.
\]
Then%
\begin{align*}
\frac{\partial F(D^{2}u)}{\partial u_{ij}}  & =\frac{\partial f(w)}{\partial
w_{kl}}\frac{\partial w_{kl}}{\partial u_{ij}}\\
& =\frac{\partial f(w)}{\partial w_{kl}}\left(  \delta_{ka,ij}\delta
^{ab}u_{bl}+u_{ka}\delta^{ab}\delta_{bl,ij}\right)  \\
& =\frac{\partial f(w)}{\partial w_{kl}}\left(  \delta_{ki}u_{jl}+u_{ki}%
\delta_{lj}\right)  \\
& =\frac{\partial f(w)}{\partial w_{il}}\delta_{jm}u_{ml}+\frac{\partial
f(w)}{\partial w_{kj}}u_{km}\delta_{im}\\
& =\frac{\partial f(w)}{\partial w_{il}}\delta_{jk}u_{kl}+\frac{\partial
f(w)}{\partial w_{kj}}u_{kl}\delta_{il}.
\end{align*}
This shows (\ref{squares}) for
\[
a^{ij,kl}=\frac{\partial f(w)}{\partial w_{il}}\delta_{jk}+\frac{\partial
f(w)}{\partial w_{kj}}\delta_{il}.
\]

\end{proof}

\bibliographystyle{amsalpha}
\bibliography{schauder}

\providecommand{\bysame}{\leavevmode\hbox to3em{\hrulefill}\thinspace}
\providecommand{\MR}{\relax\ifhmode\unskip\space\fi MR }
\providecommand{\MRhref}[2]{%
  \href{http://www.ams.org/mathscinet-getitem?mr=#1}{#2}
}
\providecommand{\href}[2]{#2}
\begin{thebibliography}{CW16}

\bibitem[CC95]{CC}
Luis~A Caffarelli and Xavier Cabr{\'e}, \emph{Fully nonlinear elliptic
  equations}, vol.~43, American Mathematical Soc., 1995.

\bibitem[CW16]{CW}
Jingyi Chen and Micah Warren, \emph{On the regularity of hamiltonian stationary
  lagrangian manifolds}, arXiv preprint arXiv:1611.02641 (2016).

\bibitem[DK11]{DongKimARMA}
Hongjie Dong and Doyoon Kim, \emph{On the lp-solvability of higher order
  parabolic and elliptic systems with bmo coefficients}, Archive for Rational
  Mechanics and Analysis \textbf{199} (2011), no.~3, 889--941.

\bibitem[Dri03]{Driver03}
Bruce. Driver, \emph{Analysis tools with applications}, 2003.

\bibitem[DZ15]{DZ}
Hongjie Dong and Hong Zhang, \emph{Schauder estimates for higher-order
  parabolic systems with time irregular coefficients}, Calc. Var. Partial
  Differential Equations \textbf{54} (2015), no.~1, 47--74. \MR{3385152}

\bibitem[GT01]{GT}
David Gilbarg and Neil~S. Trudinger, \emph{Elliptic partial differential
  equations of second order}, Classics in Mathematics, Springer-Verlag, Berlin,
  2001, Reprint of the 1998 edition. \MR{1814364}

\bibitem[HL11]{HL}
Qing Han and Fanghua Lin, \emph{Elliptic partial differential equations},
  second ed., Courant Lecture Notes in Mathematics, vol.~1, Courant Institute
  of Mathematical Sciences, New York; American Mathematical Society,
  Providence, RI, 2011. \MR{2777537}

\bibitem[MJ09]{Morrey66}
Charles~Bradfield Morrey~Jr, \emph{Multiple integrals in the calculus of
  variations}, Springer Science \& Business Media, 2009.

\bibitem[Oh93]{MR1202805}
Yong-Geun Oh, \emph{Volume minimization of {L}agrangian submanifolds under
  {H}amiltonian deformations}, Math. Z. \textbf{212} (1993), no.~2, 175--192.
  \MR{1202805 (94a:58040)}

\bibitem[Sim97]{Simon}
Leon Simon, \emph{Schauder estimates by scaling}, Calculus of Variations and
  Partial Differential Equations \textbf{5} (1997), no.~5, 391--407.

\bibitem[SW03]{SW03}
Richard Schoen and Jon Wolfson, \emph{The volume functional for {L}agrangian
  submanifolds}, Lectures on partial differential equations, New Stud. Adv.
  Math., vol.~2, Int. Press, Somerville, MA, 2003, pp.~181--191. \MR{2055848
  (2005f:53141)}

\end{thebibliography}

\end{document}